\documentclass[11pt]{article}
\usepackage{geometry}                		
\usepackage[parfill]{parskip}    		

\usepackage{amssymb}
\usepackage{amsmath}

\usepackage{palatino}

\usepackage{graphicx}
\usepackage[english]{babel}
\usepackage{amscd}
\usepackage{amsthm}

\usepackage{color}

\newtheorem{theorem}{Theorem}[section]
\newtheorem{lemma}[theorem]{Lemma}

\newtheorem{definition}{Definition}

\newtheorem{example}[theorem]{Example}

\DeclareMathOperator*{\diam}{diam}

\title{The de Bruijn-Erd\H{o}s theorem from a Hausdorff measure point of view}

\author{
Martin Dole\v{z}al \thanks{Institute of Mathematics, Czech Academy of Sciences, \v{Z}itn\'a 25, 
115 67,   Praha 1, 
Czech Republic.  
Research of Dole\v{z}al was supported by the GA\v{C}R project 17-27844S and RVO: 67985840.
E-mail:. dolezal@math.cas.cz} 
\and
Themis Mitsis\thanks{Department of Mathematics and Applied Mathematics, University of Crete, 70013 Heraklion, Greece. 
E-mail: themis.mitsis@gmail.com}
\and
Christos Pelekis\thanks{
Institute of Mathematics, Czech Academy of Sciences, \v{Z}itn\'a 25, 
115 67,   Praha 1, 
Czech Republic.  
Research supported by the GA\v{C}R project 18-01472Y and RVO: 67985840. 
E-mail: pelekis.chr@gmail.com} }

\begin{document}
	\maketitle
	
\begin{abstract}
Motivated by a well-known result in extremal set theory, due to Nicolaas Govert de Bruijn and Paul Erd\H{o}s, we consider curves in the unit $n$-cube $[0,1]^n$ of the form 
\[ A=\{(x,f_1(x),\ldots,f_{n-2}(x),\alpha): x\in [0,1]\}, \]
where $\alpha$ is a fixed real number in $[0,1]$ and $f_1,\ldots,f_{n-2}$ are injective measurable functions from $[0,1]$ to $[0,1]$. We refer to such a curve $A$ as an $n$-\emph{de~Bruijn-Erd\H{o}s-set}.  
Under the additional assumption that all functions $f_i,i=1,\ldots,n-2,$ are piecewise monotone, 
we show that the Hausdorff dimension of $A$ is at most $1$ as well as that its $1$-dimensional Hausdorff measure is at most $n-1$.  Moreover, via a walk along devil's staircases, we construct a piecewise monotone $n$-de~Bruijn-Erd\H{o}s-set whose $1$-dimensional Hausdorff measure equals $n-1$.  
\end{abstract}



\section{Prologue, related work and main results}

Here and later, $[n]$ denotes the set of positive integers $\{1,\ldots,n\}$. The collection  of all subsets of $[n]$ is denoted $2^{[n]}$. 
Given  $i\in [n]$, we denote by $\pi_i:[0,1]^n\to [0,1]^n$ the function that maps the point  
$(x_1,\ldots,x_n)$ to the point $(y_1,\ldots,y_n)$, where $y_i=x_i$, and $y_j=0$ for $j\neq i$. In other words, $\pi_i$ is the projection onto the $i$-th coordinate.  
Given a finite set $F$, we denote by $|F|$ its cardinality. Finally,  $\lambda(\cdot)$ denotes the Lebesgue measure on the real line. 

We shall be interested in an extremal problem which is motivated by a particular result from extremal set theory.  Extremal set theory (see \cite{Anderson, Engeltwo}) is concerned with the problem of obtaining sharp estimates on the cardinality of a family $\mathcal{F}\subset 2^{[n]}$ under constraints that are described in terms of union, intersection or inclusion.  This is a rapidly evolving area in combinatorics which interacts 
with various branches of mathematics and theoretical computer science including geometry, probability theory, analysis and complexity theory.  Part of this interaction is based upon the idea that several  results from  extremal set theory have continuous counterparts. 
This is an idea that dates back to the 70's and, since its conception, several results have been reported in a measurable setting (see \cite{engel, katonaone, katonatwo, klainrota}) as well as in a vector space setting (see \cite{blokhuis, Frankl_Tokushige, FranklWilson, katchalski}). 
In this article we look at results of extremal set theory from a Hausdorff dimension point of view. 
In \cite{EMP}, Konrad Engel and the last two authors reported analogues of Sperner's theorem and the Erd\H{o}s-Ko-Rado theorem in this setting. In this article, we investigate an extremal problem which is motivated by a well-known result from extremal set theory, due to de Bruijn and Erd\H{o}s (see \cite{dbruijnerdos}), which reads as follows. \\

\begin{theorem}[de~Bruijn-Erd\H{o}s]
Let $\mathcal{F}\subset 2^{[n]}$. Assume that for any two sets $A,B\in \mathcal{F}$ there exists unique $i\in [n]$ such that $i\in A\cap B$. 
Then $|\mathcal{F}|\le n$. 
\end{theorem}

See \cite{dbruijnerdos}, or \cite[Theorem 7.3.1]{cameron} for a proof of this result. Let us remark that the bound is sharp and is attained by the family $\mathcal{F}=\{[n]\setminus \{1\}, \{1,2\},\{1,3\},\ldots,\{1,n\}\}$,  which is referred to in the literature as a \emph{near-pencil}.
 
We look at the de Bruijn-Erd\H{o}s theorem from a Hausdorff measure-theoretic perspective. Before being more precise, let us briefly mention that one can identify a set $A\subset [n]$ with a binary vector of length $n$: put $1$ in the $i$-th coordinate if $i\in A$, and $0$ otherwise. Notice that this correspondence is bijective and one may choose not to distinguish between subsets and binary vectors. With this remark in mind, the de Bruijn-Erd\H{o}s theorem can be expressed by saying that if $A\subset \{0,1\}^n$ is such that for every two points  
$\mathbf{x}=(x_1,\ldots, x_n)$ and $\mathbf{y} = (y_1,\ldots,y_n)$ in $A$ there exists \emph{unique} $i\in [n]$ with $x_i=y_i>0$ then $|A|\le n$. Inspired by the last observation, we introduce the following. \\

\begin{definition}[$n$-de~Bruijn-Erd\H{o}s-sets]\label{def:1}
Fix a positive integer $n\ge 2$. 
An $n$-\emph{de~Bruijn-Erd\H{o}s-set} (or $n$-dBE-set for short) is a measurable set $A\subset [0,1]^n$ such that for any two points $\mathbf{x}=(x_1,\ldots, x_n)$ and $\mathbf{y} = (y_1,\ldots,y_n)$ in  $A$ there exists \emph{unique} $i\in [n]$ with $x_i=y_i$. 
\end{definition}
  
Given $i\in [n]$ and $\alpha\in [0,1]$, let $H_{i, \alpha}^n$ denote the hyperplane 
$\{(x_1,\ldots,x_n)\in [0,1]^n: x_i = \alpha\}$. 
It is easy to see that the $n$-dimensional Lebesgue measure of an $n$-dBE-set equals zero:  fix $\mathbf{a}=(a_1,\ldots,a_n)\in A$ and notice that $A\subset \cup_{i=1}^{n} H_{i, a_i}^n$. This means that $A$ is contained in the finite union of sets of dimension $n-1$ and therefore its $n$-dimensional Lebesgue measure is zero. Given this fact it is natural to look at the Hausdorff dimension of an $n$-dBE-set. 
Let us recall some notions from the theory of Hausdorff measures. 

If $A$ is a non-empty subset of $\mathbb{R}^n$, we denote by 
$\diam(A)$ its diameter. Fix a positive real number $s$ and, for $\delta>0$, let  
\[ \mathcal{H}_{\delta}^{s}(A) = \inf \left\{ \sum_{i}\textrm{diam}(U_i)^s: A \subset \bigcup_i U_i \; \textrm{and} \;  \textrm{diam}(U_i)\le \delta  \right\}
\]
where the infimum is taken over all at most countable covers $\{U_i\}$ of $A$ such that the diameter of each $U_i$ is at most $\delta$.
The limit $\lim_{\delta\rightarrow 0_+}\mathcal{H}_{\delta}^{s}(A)$, denoted $\mathcal{H}^{s}(A)$, is the $s$-\emph{dimensional Hausdorff measure} of $A$. The \emph{Hausdorff dimension} of $A$, denoted $\dim_H(A)$, is defined as 
\[ \dim_H(A) = \inf \left\{ s : \mathcal{H}^s(A) = 0   \right\} .  \] 
We refer the reader to \cite{Bishop_Peres, Evans_Gariepy, Falconer_1990} for excellent textbooks on the topic. 
Notice that $\mathcal{H}^0(A)$ equals the cardinality of $A$. 

In this article we focus on a particular subsystem of the system consisting of all measurable $n$-dBE-sets. 
In order to motivate this subsystem, fix $\mathbf{a}=(a_1,\ldots,a_n)\in A$ and notice that $A\subset \cup_{i=1}^{n} H_{i,a_i}^n$. 
Set $A_i:=A\cap H_{i,a_i}^n$, for $i\in [n]$. If $\mathcal{H}^0(A)=+\infty$, then it follows that 
there exists  $i\in [n]$ such that $\mathcal{H}^0(A_i)=+\infty$. 
We may assume, without loss of generality, that $i$ equals $n$. 
Now it is not difficult to see that $A_j=   \{\mathbf{a}\}$, for $j\neq n$  
and so $A\subset H^{n}_{n,a_n}$. 
Notice also that for every $j\in [n]\setminus \{n\}$ the function $\pi_j(\cdot)$ restricted to the set $A$ is injective.
Indeed, if $\mathbf{x}=(x_1,\ldots, x_n), \mathbf{y} = (y_1,\ldots,y_n)\in A$ are such that $\mathbf{x}\neq \mathbf{y}$ but $\pi_j(\mathbf{x})= \pi_j(\mathbf{y})$, then we have $x_n=y_n$ as well as $x_j=y_j$, 
which contradicts the hypothesis that $A$ is an $n$-dBE-set. In other words, the restriction of $\pi_j$ to $A$ is  invertible and its inverse is a function $\psi_j: \pi_j(A)\to [0,1]^n$ such that $\psi_j(\pi_j(A))=A$.  
Since all projections $\pi_l$, $l\neq n$, restricted to $A$ are injective, it follows that all component-functions (except the $n$th one) of $\psi_j$ are injective. 

We investigate   $n$-dBE-sets for which the corresponding component-functions are additionally "piecewise monotone" and are defined as follows. Here and later, the term \emph{monotone function} refers to a function which is either \emph{strictly increasing}, or \emph{strictly decreasing}.  \\

\begin{definition}[Piecewise monotone $n$-dBE-set]  
Let  $f:[0,1]\to [0,1]$ be an injective, measurable function. Then $f(\cdot)$ is referred to as \emph{piecewise monotone}, if there exists a countable partition of $[0,1]$ into measurable disjoint sets $\{S_j\}_{j\ge 1}$ such that for every $j$ the restriction of $f(\cdot)$ to $S_j$ is a monotone function. An $n$-dBE-set, $A$, is called \emph{piecewise monotone} if it is of the form 
\[A=\{(x,f_1(x),\ldots,f_{n-2}(x),\alpha): x\in [0,1]\}, \]
where $\alpha\in [0,1]$ is fixed and each function $f_i,i\in [n-2]$ is piecewise monotone.  
\end{definition}

It should be noted that without imposing this seemingly artificial restriction, little can be said about the Hausdorff dimension and the corresponding Hausdorff measure of an $n$-dBE-set. For example, the following, rather unexpected, result due to James Foran readily implies that the $1$-dimensional Hausdorff measure of a $3$-dBE-set  can be arbitrarily large. \\

\begin{theorem}[Foran \cite{Foran}]
For every positive real $K$,
there exists a measurable, injective function from $[0,1]$ to $[0,1]$ such that the $1$-dimensional Hausdorff measure of its graph is larger than $K$. 
\end{theorem}

Moreover, the Hausdorff dimension of an $n$-dBE-set can also be arbitrarily large.
Indeed, as previously mentioned, such a set must be contained in the finite union of hyperplanes; thus, its dimension is at most $n-1$, and this bound can be attained, as illustrated in the following example. \\

\begin{example}\label{themis}
We briefly sketch a construction of a Cantor-type set in the plane whose projections are injective and whose Hausdorff dimension equals $2$.   
For any $0<s<2$ and $\delta>0$ small, let us say that a set has property $C(\delta,s)$ if for every covering by disks $B_i$ of diameter $\delta$, we have $\sum_i\text{diam}(B_i)^s\geq1$. Let now $\delta_n$ and $\epsilon_n$ be suitable positive sequences converging to zero, and consider thin vertical rectangles $R_{1,j}$ of dimensions $\delta_1\times\epsilon_0$, with $\delta_1<<\epsilon_0$, such that their projections on the $x$ axis are pairwise disjoint and the set $A_1:=\bigcup_jR_{1,j}$ has property $C(\delta_1,s)$. Subsequently, consider thin horizontal rectangles $R'_{1,j}$ of dimensions $\delta_1\times\epsilon_1$, with $\epsilon_1<<\delta_1$, inside $A_1$ such that their projections on the $y$ axis are pairwise disjoint and the set $A_1':=\bigcup_jR'_{1,j}$ has property $C(\delta_1,s)$ as well. We next choose thin vertical rectangles $R_{2,j}$ of dimensions $\delta_2\times\epsilon_1$ inside $A_1'$ such that their projections on the $x$ axis are pairwise disjoint and the set $A_2:=\bigcup_j R_{2,j}$ has property $C(\delta_2,s)$. Then, we take thin horizontal rectangles $R_{2,j}$ of dimensions $\delta_2\times\epsilon_2$ inside $A_2$ such that their projections on the $y$ axis are pairwise disjoint and the set $A_2':=\bigcup_j R_{2,j}$ has property $C(\delta_2,s)$. Continuing this Cantor-type construction, we obtain a decreasing sequence   
\[A_1\supset A_1'\supset A_2\supset A_2'\supset\cdots\]
of sets whose intersection $A$ has Hausdorff dimension at least $s$  
and such that the projections restricted to $A$ are injective. In fact, by replacing $s$ with $s_n$ in the steps above, where $s_n$ tends to $2$, the resulting set may have dimension exactly equal to $2$. The idea of this construction was suggested by P. Mattila (personal communication with the second author).
Now, performing an analogous construction on a hyperplane results in an $n$-dBE-set of dimension $n-1$. 
\end{example}

It also seems that a combination of Foran's theorem and Example~\ref{themis} produces $3$-dBE sets whose $2$-dimensional Hausdorff measure can be arbitrarily large. 
Bearing the above in mind, our main result on the Hausdorff dimension and the corresponding Hausdorff measure of \emph{piecewise monotone $n$-dBE-sets} reads as follows. \\

\begin{theorem}\label{thm:1} 
Let $A\subset [0,1]^n$ be a measurable, piecewise monotone $n$-dBE-set. Then $\dim_H(A)\le 1$ as well as
\[\mathcal{H}^1(A) \leq n-1. \] 
\end{theorem}

We prove Theorem~\ref{thm:1} in Section~\ref{sec:2}.  
In Section~\ref{sec:4}, we show that the upper bound on the $1$-dimensional Hausdorff measure of an $n$-dBE-set, provided by Theorem \ref{thm:1}, is sharp. More precisely, we have the following. \\

\begin{theorem}\label{thm:3} 
There exists a measurable, piecewise monotone $n$-dBE-set, $A$, such that 
\[\mathcal{H}^1(A)=n-1.\] 
\end{theorem}

The set constructed in the proof of Theorem~\ref{thm:3} is a piecewise monotone $n$-dBE-set, $A$, of the form
\[ A := \{(x,h(x), f_1(h(x)), \ldots, f_{n-3}(h(x)), \alpha): x\in [0,1]\} , \]
where $h$ is a \emph{singular function} (i.e., continuous, strictly increasing, having derivative zero almost everywhre) and each $f_i,i=1,\ldots,n-3$, is a strictly increasing function from $[0,1]$ to $[0,1]$ that maps a particular set of measure zero to a set of measure one. The existence of the latter functions is guaranteed by the following result, which may be of independent interest. \\

\begin{theorem}\label{thm:4}
For every $S\subset [0,1]$ with $\lambda(S)=1$ there exists a strictly increasing function $f:[0,1]\to[0,1]$ and $N\subset S$ with $\lambda(N)=0$ such that $\lambda(f(N))=1$. 
\end{theorem}

\section{Proof of Theorem~\ref{thm:1}}\label{sec:2}

In this section we prove Theorem \ref{thm:1}.
We begin with the case of $2$-dBE-sets, which is rather trivial.  \\

\begin{theorem}
Let $A$ be a $2$-dBE-set. Then $\mathcal{H}^1(A)\le 1$. 
\end{theorem}
\begin{proof}
Let $\mathbf{x}=(x_1,x_2),\mathbf{y}=(y_1,y_2)\in A$ be distinct and notice that the hypothesis implies that either $x_1=y_1$ or $x_2=y_2$, but not both. Without loss of generality, assume that $x_1=y_1$. Hence both $\mathbf{x},\mathbf{y}$ belong to the line $\ell:=\{(x_1,\alpha): \alpha\in [0,1]\}$. Now notice that every other point from $A$ must belong to the line $\ell$ as well and thus $A\subset \ell$. The result follows. 
\end{proof}

So, from now on, we may assume  that $n\ge 3$.  
The proof of Theorem~\ref{thm:1} requires some definitions. 
We say that a family $\mathcal F$ of pairwise disjoint subsets of the real line is \emph{left-right ordered} if for every two distinct nonempty sets $F_1,F_2\in\mathcal F$ we have either $\sup F_1\le\inf F_2$ or $\inf F_1\ge\sup F_2$. We say that $\mathcal F$ is a \emph{left-right ordered partition} of a subset $A$ of the real line if it is a left-right ordered family as well as a partition of $A$.

It is well known (see \cite[p. 4]{Bishop_Peres}) that for every subset $A$ of the real line and for every $\delta,\varepsilon>0$ there is a cover $\{U_i\}_{i=1}^\infty$ of $A$ consisting of convex sets (i.e. intervals) such that the diameter of each $U_i$ is at most $\delta$ and such that  $\sum_{i}\diam(U_i)\le\mathcal{H}_{\delta}^{1}(A)+\varepsilon$. If for every $i$ we define $V_i:=(U_i\setminus\bigcup_{j<i}U_j)\cap A$ then it is clear that $\{V_i\}$ is a partition of $A$ such that the diameter of each $V_i$ is at most $\delta$. It follows that for every subset $A$ of the real line and for every $\delta,\varepsilon>0$ we have
\[ \mathcal{H}_{\delta}^{1}(A) = \inf\left\{\sum_{i}\diam(V_i)\right\}
\]
where the infimum is taken over all at most countable left-right ordered partitions $\{V_i\}$ of $A$ such that the diameter of each $V_i$ is at most $\delta$.

Recall that a function 
$f:F \subset \mathbb{R}^n\to\mathbb{R}^m$ is 
\emph{Lipschitz with constant $c$} if 
\[ |f(\mathbf{x})-f(\mathbf{y})| \leq c\cdot |\mathbf{x}-\mathbf{y}|  \; \text{for all} \; \mathbf{x},\mathbf{y}\in F . \]
The following result is well known (see \cite[p. 24]{Falconer_1990}). \\

\begin{lemma}\label{Lipschitz_lemma}
Fix positive integers $n,m$ and let $F\subset \mathbb{R}^n$. If $f: F\to\mathbb{R}^m$ is a Lipschitz function with constant $c$ then $\mathcal{H}^s(f(F))\leq c^s \mathcal{H}^s(F)$.
\end{lemma}

The proof of Theorem~\ref{thm:1} is based upon the following lemmata. \\

\begin{lemma}\label{lem:2partitions}
Let $A$ be a subset of the real line and let $\delta$ be a positive real number. Suppose that $\{V_i\}_i,\{W_j\}_j$ are at most countable left-right ordered partitions of $A$. Then $\{V_i\cap W_j\}_{i,j}$ is an at most countable left-right ordered partition of $A$ and $$\sum_{i,j}\diam(V_i\cap W_j)\le\min\left\{\sum_i\diam(V_i),\sum_j\diam(W_j)\right\}.$$
\end{lemma}

\begin{proof}
The fact that $\{V_i\cap W_j\}_{i,j}$ is an at most countable left-right ordered partition of $A$ is trivial. We will show that for every $i$ it holds
$$\sum_j\diam(V_i\cap W_j)\le\diam(V_i),$$
then the assertion easily follows. To this end, fix an index $i$ as well as $\varepsilon>0$, and find a finite set $F=\{j_1,\ldots,j_n\}$ of indices $j$ such that
$$\sum_j\diam(V_i\cap W_j)\le\sum_{k=1}^n\diam(V_i\cap W_{j_k})+\varepsilon.$$
For every $k=1,\ldots,n$, find $x_k,y_k\in V_i\cap W_{j_k}$ such that
$$\diam(V_i\cap W_{j_k})\le|x_k-y_k|+\frac\varepsilon n.$$
Without loss of generality suppose that $x_1\le y_1\le x_2\le\ldots\le y_{n-1}\le x_n\le y_n$.
Then
$$\sum_j\diam(V_i\cap W_j)\le\sum_{k=1}^n\diam(V_i\cap W_{j_k})+\varepsilon\le\sum_{k=1}^n\left(y_k-x_k+\frac\varepsilon n\right)+\varepsilon$$
$$=\sum_{k=1}^n(y_k-x_k)+2\varepsilon\le y_n-x_1+2\varepsilon\le\diam(V_i)+2\varepsilon.$$
As this holds for every $\varepsilon>0$, the proof is finished.
\end{proof}

\begin{lemma}\label{lem:measureOfTheSum}
Suppose that $f_0,f_1,\ldots,f_m$ are strictly increasing real functions, all of them defined on a given subset $D$ of the real line. Then
$$\mathcal H^1\left(\left\{\sum_{k=0}^mf_k(x)\colon x\in D\right\}\right)\le\sum_{k=0}^m\mathcal H^1\left(f_k(D)\right).$$
\end{lemma}

\begin{proof}
We may assume that $m=1$ (the general case follows by induction on $m$).
Then it clearly suffices to show that for every $\delta>0$ it holds
$$\mathcal H^1_{2\delta}\left((f_1+f_2)(D)\right)\le\mathcal H^1_\delta\left(f_1(D)\right)+\mathcal H^1_\delta\left(f_2(D)\right).$$
So fix $\delta>0$ and $\varepsilon>0$ and find at most countable left-right ordered partitions $\{V^1_i\}_i$ of $f_1(D)$ and $\{V^2_j\}_j$ of $f_2(D)$ such that the diameter of each $V^1_i$ and of each $V^2_j$ is at most $\delta$, and such that we have $$\sum_i\diam(V^1_i)\le\mathcal H^1_{\delta}(f_1(D))+\varepsilon$$
and
$$\sum_j\diam(V^2_j)\le\mathcal H^1_{\delta}(f_2(D))+\varepsilon.$$
We define $\mathcal F:=\{f_1^{-1}(V^1_i)\cap f_2^{-1}(V^2_j)\}_{i,j}$.
The fact that both functions $f_1$ and $f_2$ are increasing clearly gives us that $\mathcal F_1:=\{f_1(F)\colon F\in\mathcal F\}$ is an at most countable left-right ordered partition of $f_1(D)$. Moreover, Lemma \ref{lem:2partitions} applied to the partitions $\{V^1_i\}_i$ and $\mathcal F_1$ gives us that the partition $\mathcal F_1$ satisfies
$$\sum_{F\in\mathcal F_1}\diam(F)\le\sum_i\diam(V^1_i)\le\mathcal H^1_\delta(f_1(D))+\varepsilon.$$
Similarly $\mathcal F_2:=\{f_2(F)\colon F\in\mathcal F\}$ is an at most countable left-right ordered partition of $f_2(D)$ and
$$\sum_{F\in\mathcal F_2}\diam(F)\le\mathcal H^1_\delta(f_2(D))+\varepsilon.$$
It remains to observe that the family $\{(f_1+f_2)(F)\colon F\in\mathcal F\}$ is a cover of $(f_1+f_2)(D)$ (in fact, it is a left-right ordered partition) consisting of sets of diameter at most $2\delta$. Therefore
$$\mathcal H^1_{2\delta}((f_1+f_2)(D))\le\sum_{F\in\mathcal F}\diam((f_1+f_2)(F))=\sum_{F\in\mathcal F}\diam(f_1(F))+\sum_{F\in\mathcal F}\diam(f_2(F))$$
$$\le\mathcal H^1_\delta(f_1(D))+\mathcal H^1_\delta(f_2(D))+2\varepsilon.$$
As this holds for every $\varepsilon>0$, the proof is completed.
\end{proof}

\begin{lemma}\label{lem:monotoneCase}
Suppose that $f_0,f_1,\ldots,f_m$ are monotone real functions, all of them defined on a given subset $D$ of the real line.
Denote $G:=\{(f_0(x),f_1(x),\ldots,f_m(x))\colon x\in D\}$.
Then
$$\mathcal H^1(G)\le\sum_{k=0}^m\mathcal H^1(f_k(D)).$$
\end{lemma}

\begin{proof}
Without loss of generality, we may assume that all the functions $f_0,f_1,\ldots,f_m$ are increasing. Let us denote $S:=\{\sum_{k=0}^m f_k(x)\colon x\in D\}$. Then the mapping $F\colon S\rightarrow G$ defined by $F(\sum_{k=0}^m f_k(x)):=(f_0(x),f_1(x),\ldots,f_m(x))$ is Lipschitz with constant 1, and it is surjective. Therefore $\mathcal H^1(G)\le\mathcal H^1(S)$, by Lemma~\ref{Lipschitz_lemma}. The rest follows by Lemma~\ref{lem:measureOfTheSum}.
\end{proof}

We now have all the necessary ingredients for the proof of the main result of this section.

\begin{proof}[Proof of Theorem \ref{thm:1}] 
Notice that the first statement is a consequence of the second and so it is enough to show  the second statement. Let $A$ be of the form
\[A=\{(x,f_1(x),\ldots,f_{n-2}(x),\alpha): x\in [0,1]\}, \]
where $\alpha\in [0,1]$ is fixed and each function $f_i,i\in [n-2]$ is piecewise monotone. Define the set
\[ A' := \{ (x,f_1(x),\ldots,f_{n-2}(x)): x\in [0,1] \} \]
and notice that the fact $\mathcal{H}^1(A) = \mathcal{H}^1(A')$ implies that it is enough to show 
$\mathcal{H}^1(A') \le n-1$.  
	
Since $f_i,i\in [n-2]$, is piecewise monotone, there exists an at most countable partition, $\{S_{i,j}\}_j$, of $[0,1]$ into measurable sets such that the function $f_i$ restricted to each $S_{i,j}$ is monotone. For every $(n-2)$-tuple $\mathbf{j}=(j_1,\ldots,j_{n-2})$ let 
$S_{\mathbf{j}} = S_{1,j_1}\cap \cdots \cap S_{n-2,j_{n-2}}$ and denote by 
$f_{i,\mathbf{j}}$ the restriction of $f_i$ to the set $S_{\mathbf{j}}$. Notice that the sets $\{S_{\mathbf{j}}\}_{\mathbf{j}}$ are pairwise disjoint and that each function $f_{i,\mathbf{j}}$ is monotone.
For every $\mathbf j$ denote 
\[G_{\mathbf j}:=\{(x,f_1(x),\ldots,f_{n-2}(x))\colon x\in S_{\mathbf j}\}.\]
Then, by Lemma~\ref{lem:monotoneCase} applied to the monotone functions $f_{0,\mathbf j}:=\textrm{id}|_{S_{\mathbf j}}$ and $f_{1,\mathbf j},\ldots,f_{n-2,\mathbf j}$ (for every $\mathbf j$), we have
\begin{eqnarray*}\mathcal H^1(A')&=&\sum_{\mathbf j}\mathcal H^1(G_{\mathbf j})\le\sum_{\mathbf j}\left(\mathcal H^1(S_{\mathbf j})+\sum_{k=1}^{n-2}\mathcal H^1(f_k(S_{\mathbf j}))\right) \\
&=&\mathcal H^1([0,1])+\sum_{k=1}^{n-2}\mathcal H^1(f_k([0,1]))\le n-1, 
\end{eqnarray*}
as desired. 
\end{proof}

\section{Proof of Theorem~\ref{thm:3} and Theorem~\ref{thm:4}}\label{sec:4}

In this section we prove Theorem \ref{thm:3}. We begin with the case $n=3$. 
In this case the result is, essentially, proven in \cite{EMP}, in the context of obtaining a continuous analogue of Sperner's theorem, but we repeat here the proof for the sake of completeness. 

The proof employs the following result from measure theory (see \cite[Proposition 5.5.4]{Bogachev}). \\

\begin{lemma}\label{bogachev} Let $f:[0,1]\to [0,1]$ be a function and let $E$ be a measurable set such that at every point of $E$ the function $f(\cdot)$ is differentiable. 
Then 
\[ \lambda(f(E)) \leq \int_E |f'(x)|\; dx . \]
\end{lemma} 

We can now proceed with the proof for the existence of a $3$-dBE-set whose $1$-dimensional Hausdorff measure equals $2$. 

 \begin{proof}[Proof of Theorem \ref{thm:3} (case $n=3$)]
Let $h:[0,1]\to[0,1]$ be a continuous, strictly increasing function having zero derivative almost everywhere. 
An example of such a function can be found in \cite{Zaanen_Lux}. 
Now fix $\alpha \in [0,1]$ and define the set 
\[ A = \{ (x,h(x), \alpha) : x\in [0,1] \} . \]
Clearly, $A$ is a $3$-dBE-set and it is enough to show that $\mathcal{H}^1(A)=2$. 
Since $\mathcal{H}^1(A) = \mathcal{H}^1(\{(x,h(x)):x\in[0,1]\})$ it is enough to show that the set 
$D :=\{(x,h(x)):x\in[0,1]\}$ satisfies 
$\mathcal{H}^1(D)=2$. 

Divide $D$ into two parts, namely, 
$D_1=\{(x,h(x)): h'(x)=0\}$ and $D_2 = D\setminus D_1$. 
Since $h'=0$ almost everywhere, the projection of $D_1$ onto the $x$-axis has measure $1$ and so, by Lemma \ref{Lipschitz_lemma}, we have $\mathcal{H}^1(D_1) \ge 1$. 
From Lemma \ref{bogachev}, it follows that the set $h(\{x:h'(x)=0\})$ has measure zero and therefore it follows that the projection of $D_2$ onto the $y$-axis has measure $1$. Thus, by Lemma \ref{Lipschitz_lemma}, $\mathcal{H}^1(D_2) \ge 1$.
Putting these two bounds together, we conclude 
\[ \mathcal{H}^1(D) = \mathcal{H}^1(D_1) + \mathcal{H}^1(D_2) \ge 2  \]
and therefore, using Theorem~\ref{thm:1}, we have $\mathcal{H}^1(D)=2$. The result follows.  
\end{proof}

The proof of Theorem \ref{thm:3}, for $n\ge 4$, is based upon Theorem~\ref{thm:4}. 
Hence we now proceed with the proof of Theorem~\ref{thm:4}, which requires the following lemma. In the proof, $\omega$ denotes the set of all nonnegative integers, $2^\omega$ the set of all infinite dyadic sequences, $2^{<\omega}$ the set of all finite dyadic sequences, and $2^{\leq n}$ the set of all dyadic sequences of length at most $n$ and $2^n$ the set of all dyadic sequences of length equal to $n$ . Finally, given $s=(s_1,\ldots, s_n)\in 2^n$ and $i\in \{0,1\}$, we denote by $s^\wedge\{i\}$  the dyadic sequence $(s_1,\ldots,s_n,i)$.  \\

\begin{lemma}\label{claim:one_interval}
	Let $I=[a,b]\subset[0,1]$ be a closed interval and let $S_I\subset[0,1]$ be such that $\lambda(I\setminus S_I)=0$. Then there exists a set $N_I\subset S_I\cap I$ and a continuous non-decreasing function $f_I\colon I\rightarrow[0,1]$ such that $f_I(a)=0$, $f_I(b)=1$, $\lambda(N_I)=0$ and $\lambda(f_I(N_I))=1$.
\end{lemma}

\begin{proof}
	By assumption we have $\lambda(S_I\cap I)=\lambda(I)>0$ and the inner regularity of the Lebesgue measure implies that the set $S_I\cap I$ contains a closed subset $F$ with $\lambda(F)>0$. In particular, $F$ is uncountable. As every closed set has the perfect set property (see~\cite{Kechris}), there is a nonempty perfect (i.e. closed and with no isolated points) subset $\tilde F$ of $F$. By induction on the length of $s$, we will construct closed intervals $J_s$, $s\in 2^{<\omega}$, such that for every $n\in\omega$, the following conditions hold:
	\begin{itemize}
		\item[(i)$_n$] $J_{s^\wedge\{0\}}\cup J_{s^\wedge\{1\}}\subset J_s$ for every $s\in 2^n$,
		\item[(ii)$_n$] the upper endpoint of $J_{s^\wedge\{0\}}$ is strictly below the lower endpoint of $J_{s^\wedge\{1\}}$ for every $s\in 2^n$,
		\item[(iii)$_n$] $\lambda(J_s)\le\frac 1{(n+1)2^n}$ for every $s\in 2^n$,
		\item[(iv)$_n$] the interior of $J_s$ intersects $\tilde F$ for every $s\in 2^n$.
	\end{itemize}
Simply begin the construction by setting $J_\emptyset=I$; then both conditions (iii)$_0$ and (iv)$_0$ are satisfied. Now suppose that we have already defined all intervals $J_s$, $s\in 2^{\le n}$, for some $n\in\omega$ such that conditions (i)$_k$, (ii)$_k$ hold for every $k<n$ and such that conditions (iii)$_k$, (iv)$_k$ hold for every $k\le n$. By $(iv)_n$ we know that the interior of $J_s$ intersects $\tilde F$ for every $s\in 2^n$. As $\tilde F$ is a perfect set, the intersection $J_s\cap\tilde F$ is an infinite (even uncountable) set for every $s\in 2^n$. So for every $s\in 2^n$, we can find two points $r_{s^\wedge\{0\}}<r_{s^\wedge\{1\}}$ in $J_s\cap\tilde F$. Then for every $s\in 2^n$, we can find closed intervals $J_{s^\wedge\{0\}}$ and $J_{s^\wedge\{1\}}$ containing $r_{s^\wedge\{0\}}$ and $r_{s^\wedge\{1\}}$, respectively, in such a way that conditions (i)$_n$, (ii)$_n$ and (iii)$_{n+1}$ are satisfied. Then condition (iv)$_{n+1}$ clearly holds as well. This finishes the construction.
	
Let $N_I$ be the subset of $I$ given by $N_I=\bigcap_{n\in\omega}\bigcup_{s\in 2^n}J_s$. Then for every $m\in\omega$, we have by condition (iii)$_m$ that
	\begin{equation*}
		\lambda(N_I)\le\lambda\left(\bigcup_{s\in 2^m}J_s\right)\le\sum_{s\in 2^m}\lambda(J_s)\le\frac 1{m+1},
	\end{equation*}
which implies that $\lambda(N_I)=0$. Also, for every $x\in N_I$ and every $n\in\omega$ there is $s\in 2^n$ such that $x\in J_s$, and by (iii)$_{n}$ and (iv)$_n$ there is $y\in\tilde F$ such that $|x-y|\le\frac 1{(n+1)2^n}$. It follows that every $x\in N_I$ is in the closure of the closed set $\tilde F$, and so $N_I\subset\tilde F\subset F\subset S_I\cap I$.
	
For every $x\in N_I$ there is (by (i)$_n$ and (ii)$_n$) a unique $\alpha_x\in 2^\omega$ such that $x\in\bigcap_{n\in\omega}J_{\alpha_x|_n}$. Moreover, if $x,y$ are two distinct points from $N_I$ then $\alpha_x\neq\alpha_y$ by (iii)$_n$. It clearly follows that the mapping $\phi\colon N_I\rightarrow[0,1]$ which maps each $x\in N_I$ to the unique point from $\bigcap_{n=1}^\infty[\frac{2\alpha_x(n)}{3^n},\frac{2\alpha_x(n)+1}{3^n}]$ is an order preserving homeomorphism of the set $N_I$ onto the standard ternary Cantor set $C\subset[0,1]$. Let $\Phi\colon I\rightarrow[0,1]$ be the continuous non-decreasing extension of $\phi$ which is linear on each open subinterval of $I\setminus N_I$. Let $c\colon[0,1]\rightarrow[0,1]$ be the Devil's staircase  (see \cite{DMRV}), i.e. a continuous non-decreasing function with $c(0)=0$ and $c(1)=1$ which is constant on every open subinterval of $[0,1]\setminus C$, and so $\lambda(c(C))=1$. Then we define $f_I=c\circ\Phi$. This is clearly a continuous and non-decreasing function with $f_I(a)=c(0)=0$, $f_I(b)=c(1)=1$ and $\lambda(f_I(N_I))=\lambda(c(C))=1$.
\end{proof}

We can now prove Theorem~\ref{thm:4}.

\begin{proof}[Proof of Theorem~\ref{thm:4}]
Let $S\subset[0,1]$ be such that $\lambda(S)=1$. Let $I_n=[a_n,b_n]$, $n\in\omega$, be an enumeration of all closed subintervals of $[0,1]$ with rational endpoints. By Lemma~\ref{claim:one_interval}, we can find (by induction on $n$) sets $N_{I_n}\subset S\setminus\bigcup_{m<n}N_{I_m}$, $n\in\omega$, and continuous non-decreasing functions $f_{I_n}\colon I_n\rightarrow[0,1]$, $n\in\omega$, such that $f_{I_n}(a_n)=0$, $f_{I_n}(b_n)=1$, $\lambda(N_{I_n})=0$ and $\lambda(f_{I_n}(N_{I_n}))=1$ for every $n\in\omega$. 

For every $n\in\omega$, let $g_n\colon[0,1]\rightarrow[0,1]$ be the unique continuous non-decreasing extension of $f_{I_n}$ from $I_n$ to the whole interval $[0,1]$. Finally, we define $N=\bigcup_{n\in\omega}N_{I_n}$ and 
\[ f(x)=\sum\limits_{n\in\omega}\frac 1{2^{n+1}}g_n(x),\; \text{for}\; x\in[0,1].\] 
Clearly, $N\subset S$ and $\lambda(N)=0$. Since the sum in the definition of $f$ converges uniformly and all summands are continuous functions, it follows that $f$ is continuous as well. Also, $f$ is the sum of non-decreasing functions and so it is non-decreasing as well. Moreover, if $x<y$ are two points from $[0,1]$ then there exists $m\in\omega$ such that $x<a_m<b_m<y$. Now notice that 
\begin{equation*}
	f(y)-f(x)=\sum\limits_{n\in\omega}\frac 1{2^{n+1}}(g_n(y)-g_n(x))\ge\frac 1{2^{m+1}}(g_m(y)-g_m(x))=\frac 1{2^{m+1}}>0,
\end{equation*}
which implies that $f$ is strictly increasing. It is obvious that $f(0)=0$ and $f(1)=1$, so it only remains to show that $\lambda(f(N))=1$. 

To this end, let us fix $m\in\omega$ for a while. If $N_{I_m}$ and $f_{I_m}$ were constructed in the same way as in the proof of Lemma~\ref{claim:one_interval}, then $N_{I_m}$ is closed and $g_m$ is constant on every open subinterval of $[0,1]\setminus N_{I_m}$. So if $f^{(m)}\colon[0,1]\rightarrow[0,1]$ is the (strictly increasing and continuous) function given by 
\[ f^{(m)}(x)=\sum\limits_{\substack{n\in\omega\\n\neq m}}\frac 1{2^{n+1}}g_n(x),\; \text{for}\; x\in[0,1],\] 
then for every open subinterval $(c,d)$ of $[0,1]$ which does not intersect $N_{I_m}$ we have
\begin{equation*}
	\lambda(f(c,d))=f(d)-f(c)=f^{(m)}(d)-f^{(m)}(c)=\lambda(f^{(m)}((c,d))).
\end{equation*}
Since the set $[0,1]\setminus N_{I_m}$ is the union of countably many pairwise disjoint such subintervals $(c_i,d_i)$, $i\in\omega$, and of a subset of $\{0,1\}$ it follows that
\begin{equation}\label{eq:image_of_complement}
\begin{split}
\lambda(f([0,1]\setminus N_{I_m}))&=\sum\limits_{i\in\omega}\lambda(f((c_i,d_i)))=\sum\limits_{i\in\omega}\lambda(f^{(m)}((c_i,d_i)))\\
&=\lambda(f^{(m)}([0,1]\setminus N_{I_m}))\le\lambda(f^{(m)}([0,1]))\\
&=f^{(m)}(1)-f^{(m)}(0)=1-\frac 1{2^{m+1}}.
\end{split}
\end{equation}
On the other hand,
\begin{equation}\label{eq:image}
\lambda(f([0,1]))=f(1)-f(0)=1.
\end{equation}
Comparing equations (\ref{eq:image_of_complement}) and (\ref{eq:image}) yields $\lambda(f(N_{I_m}))\ge\frac 1{2^{m+1}}$.
Since this is true for every $m\in\omega$ and all the sets $N_{I_m}$, $m\in\omega$, are pairwise disjoint, the monotonicity of $f$ implies   
\begin{equation*}
	\lambda(f(N))=\sum_{m\in\omega}\lambda(f(N_{I_m}))\ge\sum_{m\in\omega}\frac 1{2^{m+1}}=1,
\end{equation*}
as required.
\end{proof}

We now have all ingredients for the proof of Theorem~\ref{thm:3}, when $n\ge 4$. 

\begin{proof}[Proof of Theorem~\ref{thm:3} (case $n\ge 4$)]
Let $h:[0,1]\to[0,1]$ be a singular function and let 
$S_1:= h(\{x: h'(x)\neq 0\})$. Notice that Lemma~\ref{bogachev} implies that $\lambda(S_1)=1$. 
By Theorem~\ref{thm:4} there exists a strictly increasing function $f_1:[0,1]\to [0,1]$ and a set $N_1\subset S_1$ with $\lambda(N_1)=0$ such that $\lambda(f_1(N_1))=1$. 
Now for $j=2,\ldots,n-3$ define recursively $S_j=S_{j-1}\setminus N_{j-1}$ and apply Theorem~\ref{thm:4} 
to choose a strictly increasing function $f_j:[0,1]\to [0,1]$ and a set $N_j\subset S_j$ with $\lambda(N_j)=0$ such that $\lambda(f_j(N_j))=1$. Let $\alpha\in [0,1]$ be fixed, and consider the set 
\[ A := \{(x,h(x), f_1(h(x)), \ldots, f_{n-3}(h(x)), \alpha): x\in [0,1]\} . \]
Since the functions $h,f_1,\ldots,f_{n-3}$ are strictly increasing it follows that $A$ is an $n$-dBE-set. Moreover, it is enough to show that 
\[\mathcal{H}^1\left(\{ (x,h(x), f_1(h(x)), \ldots, f_{n-3}(h(x)): x\in [0,1])\} \right)=n-1 .\]
To this end, we proceed as in the proof of the case $n=3$ and show that the set 
\[ D = \{ (x,h(x), f_1(h(x)), \ldots, f_{n-3}(h(x)): x\in [0,1])\} \] 
can be written as the disjoint union of $n-1$ sets that project onto a particular coordinate to sets of measure one. 
Consider the sets 
\[D_1 := \{(x,h(x), f_1(h(x)), \ldots, f_{n-3}(h(x))): h'(x)=0\}\quad \text{and}\quad D_2 = D\setminus D_1.\]
As in the proof of the case $n=3$ it follows that $\mathcal{H}^1(D_1)\ge 1$. We now look at the set $D_2$. 
Let $W_j = h^{-1}(N_j)$, for  $j=1,\ldots, n-3$, and notice that $\lambda(f_j(N_j))=\lambda(f_j(h(W_j)))=1$. 
Since $h(\cdot)$ is injective, it follows that we can write 
\[S_{n-3} = h\big(\{x:h'(x)\neq 0\}\setminus \{W_1\cup\cdots  \cup  W_{n-4}\}\big) . \]
Moreover, $\lambda(S_{n-3})=1$. Thus, denoting   
$Q_1 := \{x:h'(x)\neq 0\}\setminus (W_1\cup\cdots \cup W_{n-3})$,   
it follows that the set $D_2$ can be written as the disjoint union of the sets $D_{2,1},D_{2,2},\ldots, D_{2,n-2}$, where 
\[ D_{2,j}= \{(x,h(x), f_1(h(x)), \ldots, f_{n-3}(h(x)): x\in W_j\}, \; \text{for}\; j=1,\ldots,n-3 \]
and 
\[ D_{2,n-2} = \{(x,h(x), f_1(h(x)), \ldots, f_{n-3}(h(x)): x\in Q_1\} . \]
Since $\lambda(h(Q_1))=1$ it follows that the set $D_{2,n-2}$ projects onto the second coordinate to a set of measure one; hence $\mathcal{H}^1(D_{2,n-2})\ge 1$. Similarly for  $j=1,\ldots,n-3$, 
since $\lambda(f_j(h(W_j)))=1$ it follows that the set 
$D_{2,j}, j=1,\ldots, n-3$, projects onto the $(j+2)$-th coordinate to a set of measure one; hence 
$\mathcal{H}^1(D_{2,j})\ge 1$ as well. 
Thus 
\[ \mathcal{H}^1(D) = \mathcal{H}^1(D_1) + \sum_{j=1}^{n-2}\mathcal{H}^1(D_{2,j}) \ge n-1 .\]
Theorem~\ref{thm:1} finishes the proof.   
\end{proof}

\end{document}